\documentclass[12pt]{amsart}
\usepackage{tikz-cd}
\usetikzlibrary{matrix,arrows,decorations.pathmorphing}
\usepackage{amssymb}
\usepackage{amsmath,amscd}
\usepackage{mathrsfs}
\usepackage{url}
\usepackage{fullpage}
\usepackage{hyperref}
\usepackage{verbatim}
\usepackage{comment}
\usepackage{fancyvrb}
\usepackage{fvextra}
\newtheorem{thm}{Theorem}[section]
\newtheorem{prop}[thm]{Proposition}
\newtheorem{lem}[thm]{Lemma}

\newtheorem{conj}[thm]{Conjecture}

\theoremstyle{definition}

\newtheorem{example}[thm]{Example}

\newtheorem{rem}[thm]{Remark}

\numberwithin{equation}{section}

\renewcommand{\L}{\mathcal{L}}

\newcommand{\T}{\mathcal{T}}

\newcommand{\qq}{\mathbb{Q}}

\newcommand{\C}{\mathcal{C}}
\newcommand{\M}{\mathcal{M}}
\renewcommand{\L}{\mathcal{L}}
\newcommand{\p}{\mathbb{P}}
\newcommand{\pp}{\mathbb{P}}

\renewcommand{\P}{\mathcal{P}}

\renewcommand{\O}{\mathcal{O}}

\renewcommand{\tilde}{\widetilde}
\DeclareMathOperator{\rank}{rank}

\DeclareMathOperator{\codim}{codim}

\DeclareMathOperator{\area}{area}
\DeclareMathOperator{\nonhyp}{nonhyp}

\usepackage{colonequals}

\newcommand{\K}{\mathcal{K}}

\author{Dawei Chen}
\address{Department of Mathematics, Boston College, Chestnut Hill, MA 02467}
\email{dawei.chen@bc.edu}

\author{Hannah Larson}
\address{Department of Mathematics, University of California Berkeley}
\email{hlarson@berkeley.edu}

\title{Independence of tautological classes and cohomological stability for strata of differentials}

\begin{document}

\begin{abstract}
The tautological rings of strata of differentials are known to be generated by divisor classes. In this paper, we give lower bounds on the degrees of relations among them, depending on the genus $g$ and the number of simple zeros. For strata with more than $4g/3$ simple zeros, our results show that there are no relations in degrees less than $\lfloor g/3 \rfloor + 1$. 
Moreover, we conjecture that, outside of a few exceptions, there is always a non-trivial relation in degree $\lfloor g/3 \rfloor + 1$, and prove the conjecture for all strata of holomorphic abelian differentials with $g \leq 30$. 
We also prove that the cohomology rings of strata of holomorphic differentials with sufficiently many simple zeros stabilize to the free algebra on the tautological divisor class. Finally, we show that for a large class of holomorphic abelian strata, containing hyperelliptic differentials, the tautological ring is non-trivial for sufficiently large $g$. 
\end{abstract}

\maketitle

\section{Introduction}

Holomorphic abelian differentials on smooth and compact complex algebraic curves induce translation structures on the underlying (Riemann) surfaces with conical singularities at the zeros of the differentials, and affine transformations in ${\rm GL}_2^{+}(\mathbb R)$ acting on them preserve the orders of the zeros. This viewpoint makes the study of moduli spaces of holomorphic differentials with prescribed orders of zeros become an important subject in moduli theory and surface dynamics. Additionally, meromorphic differentials and higher-order differentials (as possibly meromorphic sections of $K^{\otimes \ell}$ for $\ell > 1$) naturally appear for compactifying moduli spaces of holomorphic differentials and for the study of cone metrics on surfaces, respectively. In general, moduli spaces of (possibly meromorphic and higher-order) differentials with prescribed orders of zeros and poles are said to be {\em strata of differentials}, since their adjacency is induced by merging the zeros or poles, thus providing a stratification for the total space of differentials. We refer the reader to \cite{Zo06, Wr15, Ch17, BCGGMk} for an introduction to this fascinating topic. 

Up to projectivization, strata of differentials are moduli spaces parameterizing smooth curves together with  pluricanonical divisors having given multiplicity. Given $\ell \geq 1$ and an integer partition $\mu = (m_1, \ldots, m_n)$ of $\ell(2g - 2)$ with $2g-2+n > 0$, we define
\[\P^\ell(\mu) \colonequals \left\{(C,z_1, \ldots, z_n) : \sum_{i=1}^n m_i z_i \sim 
\ell K_C\right\} \subset \M_{g,n}. \]
In $\P^\ell(\mu)$, the zeros and poles are labeled, but one may also consider unlabeled strata $\P(\mu)/\mathbb{S}_\mu$, where
$\mathbb{S}_\mu \subset \mathbb{S}_n$ is the subgroup of permutations $\sigma$ satisfying $m_{\sigma(i)} = m_i$. In the holomorphic case, the unlabeled strata give stratification of the $\ell$-Hodge bundle. 

The strata $\P^\ell(\mu)$ are smooth orbifolds which nevertheless can be disconnected for special $\mu$. The classification of connected components of $\P^\ell(\mu)$ has
been completed for holomorphic abelian differentials ($\ell=1$ and $m_i \geq 0$ for all $i$, \cite{KZ}); meromorphic abelian differentials ($\ell=1$ and some $m_i < 0$, \cite{Bo15}); quadratic differentials of finite area ($\ell=2$ and $m_i \geq -1$ for all $i$, \cite{La04H, La04S, La08, CM14}); and quadratic differentials of infinite area ($\ell=2$ and some $m_i < -1$, \cite{CG22}). For $\ell\geq 3$, the classification of connected components of
$\P^\ell(\mu)$ remains unknown. Besides certain ad hoc structures occurring in low genus, the only known invariants that can  distinguish connected components for general $g$ 
are the hyperelliptic structure and the spin parity. 

On the other hand, the dimensions of components of $\P^{\ell}(\mu)$ are completely known. Let $\P^{\ell}(\mu)^{\circ}$ be a connected component of $\P^{\ell}(\mu)$. We say that $\P^{\ell}(\mu)^{\circ}$ is of {\em holomorphic abelian type} if it parameterizes $\ell$th powers of holomorphic abelian differentials. Then we have 
\[  
\dim \P^{\ell}(\mu)^{\circ} = \begin{cases}
    2g-2 + n & \mbox{if it is of holomorphic abelian type}  \\
    2g-3+n & \mbox{otherwise}
\end{cases}
\]
(see \cite{Ma82, Ve82, Ve90, BCGGMk}). 
Here, we study the Chow rings of $\P^{\ell}(\mu)$, always working with rational coefficients. 

In the study of the Chow ring of $\M_{g,n}$, certain natural classes called \emph{tautological classes} play an important role. To define these classes, let $\pi \colon \C \to \M_{g,n}$ be the universal curve, let $\sigma_1, \ldots, \sigma_n\colon \M_{g,n} \to \C$ be the disjoint sections corresponding to the $n$ markings, and let $D_i \subset \C$ denote the image of the $i$th section.
The following classes are defined in the Chow ring $A^*(\M_{g,n})$:
\[\psi_i \colonequals \sigma_i^*c_1(\omega_\pi) \in A^1(\M_{g,n}) \qquad \text{and} \qquad \kappa_j \colonequals \pi_*(c_1(\omega_\pi(D_1 + \cdots + D_n)^{j+1})) \in A^j(\M_{g,n}).\]
The subring of the Chow ring of $\M_{g,n}$ generated by the above $\psi$ and $\kappa$ classes is called the \emph{tautological ring} and is denoted $R^*(\M_{g,n}) \subseteq A^*(\M_{g,n})$.

It turns out that the generators of the tautological ring are independent in low degrees. More precisely, by \cite{Boldsen} the surjection
\[\qq[\kappa_1, \kappa_2, \ldots, \psi_1, \ldots, \psi_n] \to R^*(\M_{g,n})\]
is actually an isomorphism in degrees $* \leq g/3$.
This bound is sharp, as there are known relations in degree $\lfloor g/3 \rfloor + 1$ \cite{Ionel2}.
Moreover, stability theorems of Harer and Madsen--Weiss say that the images of tautological classes under the cycle class map freely generate the stable cohomology of $\M_{g,n}$.
In this paper, we prove analogues of each of these celebrated results for the tautological rings of strata of differentials, whose definition we recall below.

One source of tautological classes  on $\P^\ell(\mu)$ is the pullbacks of tautological classes on $\M_{g,n}$ along the inclusion 
\[\iota \colon \P^\ell(\mu) \hookrightarrow \M_{g,n}.\] 
There is also another tautological class on $\P^\ell(\mu)$ defined as follows.
Let $\mathcal{A}$ be the line bundle on $\P^\ell(\mu)$
whose fiber at $(C, z_1, \ldots, z_n)$ is the one-dimensional vector space of differentials with multiplicity $m_i$ at $z_i$.
Precisely, if
$\pi^\mu \colon \C^\mu \to \P^\ell(\mu)$ is the restriction of the universal curve and $D_i^\mu \subset \C^\mu$ are the images of the $i$th sections, then $\mathcal{A} = \pi^\mu_*\omega^{\otimes\ell}(-m_1D_1^\mu - \cdots - m_nD_n^\mu)$.
We define
\begin{equation} \label{etadef} \eta \colonequals c_1(\mathcal{A}) \in A^1(\P^\ell(\mu)).
\end{equation}
The \emph{tautological ring}, denoted $R^*(\P^{\ell}(\mu)) \subseteq A^*(\P^{\ell}(\mu))$, is defined as the subring generated by $\eta$ and the pullbacks $\iota^*\psi_i$ and $\iota^*\kappa_j$.

In \cite[Proposition 2.1]{ChenTauto}, the first named author 
proves the following relations among tautological classes
\begin{equation}  \label{iota} \eta = (m_i + \ell) \iota^*\psi_i 
\qquad \text{and} \qquad \iota^*\kappa_j = \ell^{-j}(2g - 2 + n) \eta^j.
\end{equation}
In particular, if $m_i \neq -\ell$ for all $i$, we find that $\eta$ generates the tautological ring. However, if $m_i = -\ell$ for some $i$, then $\eta = \kappa_j = 0$ and the tautological ring is generated by the $\psi$ classes at the marked points with $m_i = -\ell$.

\begin{thm} \label{thm:lowerbound}
Let $\mu = (m_1, \ldots, m_k, 1^{\ell(2g - 2) - m})$ with $m = m_1 + \cdots + m_k$ the sum of $k$ specified orders such that the remaining orders are all equal to $1$.
Let $r$ be such that $m_i < 0$ for $i \leq r$ and $m_i \geq 0$ for $i > r$.

If $m_i \neq -\ell$ for all $i$, then the surjection
\begin{equation} \label{t1}
\qq[\eta] \rightarrow R^*(\P^{\ell}(\mu))
\end{equation}
is an isomorphism in degrees $* \leq \min\{g/3, \ell(2g - 2) - m - g + \delta_{0r}\delta_{1\ell} - 1\}$.

If some $m_i = -\ell$, let $\{i_1, \ldots, i_s\} \subset \{1, \ldots, n\}$ be the  indices for which $m_{i_j} = -\ell$. Then, the surjection
\begin{equation} \label{t2}
\qq[\psi_{i_1}, \ldots, \psi_{i_s}] \rightarrow R^*(\P^{\ell}(\mu))
\end{equation}
is an isomorphism in degrees $* \leq \min\{g/3, \ell(2g - 2) - m - g + \delta_{0r}\delta_{1\ell} - 1\}$.
\end{thm}

Theorem \ref{thm:lowerbound}
is a consequence of Theorem \ref{purewt} which shows that the appropriate tautological classes freely generate the pure weight cohomology of $\P^\ell(m_1, \ldots, m_k, 1^{\ell(2g - 2) - m})/\mathbb{S}_{\ell(2g - 2) - m}$ in the specified range of degrees.
To prove this result, we introduce a partial compactification $\overline{X}^\ell(m_1, \ldots, m_k) \supset \P^\ell(m_1, \ldots, m_k, 1^{\ell(2g - 2) - m})/\mathbb{S}_{\ell(2g - 2) - m}$, which parameterizes pluricanonical divisors on $k$-pointed curves having orders at least $m_1, \ldots, m_k$ at the marked points. 
In Section \ref{sec:lower}, we prove that, away from loci of high codimension, $\overline{X}^\ell(m_1, \ldots, m_k)$ is isomorphic to a projective bundle over an open substack of $\M_{g,k}$ and give bounds on the codimension of the loci that must be removed.
Using known results on the stability of the cohomology of $\M_{g,k}$, this determines the cohomology of $\overline{X}^\ell(m_1, \ldots, m_k)$ up to high codimension.
Then, in Section \ref{sec:pure}, we study the boundary of our compactification, which it turns out is surjected onto by a union of other 
$\overline{X}^\ell(m_1', \ldots, m_{k'}')$. This allows us to set up an excision calculation which determines all relations among tautological classes on $\P^\ell(m_1, \ldots, m_k, 1^{\ell(2g - 2) - m})/\mathbb{S}_{\ell(2g - 2) - m}$ in low degrees.

In Section \ref{sec:stab}, via a more careful analysis of this stratification and the localization long exact sequence,
we prove that the entire cohomology of the strata of holomorphic $\ell$-differentials also stabilizes as the genus tends to infinity (where the extra simple zeros added are unordered). 
\begin{thm} \label{thm:holstab}
Let $m_1, \ldots, m_k > -\ell$ be given.
The cohomology of the associated strata of $\ell$-differentials with sufficiently many simple zeros stabilizes as 
\[\lim_{g \to \infty} H^*(\P^\ell(m_1, \ldots, m_k, 1^{\ell(2g - 2) - m})/\mathbb{S}_{\ell(2g - 2) - m}) = \qq[\eta].\]
\end{thm}
\noindent
See Lemma \ref{12} for more precise statements and Remark \ref{rem:gluing} for a geometric interpretation of the stabilization maps. We remark that the condition $m_i > -\ell$ for all $i$ is equivalent to the condition that the flat surface structure induced by the differential has finite area. Therefore, it makes sense to still say that an $\ell$-differential is ``holomorphic'' if its pole orders are bounded by $\ell-1$. 

Meanwhile, for strata of meromorphic differentials, we demonstrate (see Example \ref{ex:mero}) that $H^1(\P^\ell(m_1, \ldots, m_k, 1^{\ell(2g - 2) - m})/\mathbb{S}_{\ell(2g - 2) - m})$ may be non-vanishing, even as $g$ tends to infinity, so in particular the full cohomology ring cannot stabilize to the tautological ring.

\begin{rem}
The $\ell = 1$ case of Theorem \ref{thm:holstab} was independently established in contemporaneous work of Tosteson \cite{Tosteson} using different techniques. When $\ell = 1$, our stable range is $* \leq \frac{1}{2}(g - m-2)$, which is roughly $\frac{1}{6}g$ larger than the stable range there. Tosteson's work also gives stability results for the integral homology groups of $\P^1(m_1, \ldots, m_k, 1^{2g - 2 - m})/\mathbb{S}_{2g - 2 - m}$.
\end{rem}

We expect that the bounds in Theorem \ref{thm:lowerbound} are sharp when $m_i \neq -\ell$ for all $i$, and the minimum in the statement coincides with $g/3$, equivalently when the number of simple zeros is greater than $4g/3$.
This is because there are known relations in $R^*(\M_{g,n})$
beginning in degree $\lfloor g/3 \rfloor + 1$. Pulling back these relations along the inclusion $\iota$ gives rise to relations in the tautological ring of $\P^\ell(\mu)$, which we conjecture to be non-trivial outside a small list of explicit exceptions. (In the case that $m_i = -\ell$ for some $i$, the pullback of these relations is trivial by \eqref{iota}.)

\begin{conj} \label{conj:upperbound}
Let $g \geq 2$ and let $\mu$ be any partition of $\ell(2g - 2)$ such that $m_i \neq -\ell$ for all $i$. Then the kernel of \eqref{t1} is non-trivial in degree $* = \lfloor g/3 \rfloor + 1$.
\end{conj}

Combining explicit formulas for degree $\lfloor g/3 \rfloor + 1$ relations in $R^*(\M_g)$ with \eqref{iota}, in Section \ref{sec:upper}, we reduce Conjecture \ref{conj:upperbound} for holomorphic abelian differentials to the non-vanishing of a particular coefficient of an explicit power series. Using a computer, we verify this coefficient is non-vanishing in many cases.\footnote{Ranging over all partitions of $2g - 2$ for $g \leq 30$ represents the first approximately $2.6$ million cases of the conjecture.} 

\begin{thm} \label{thm:upperbound}
If $g \leq 30$, then
Conjecture \ref{conj:upperbound} holds for any positive partition of $2g - 2$.
\end{thm}

Another interesting feature is that allowing fractional and negative parts in a ``partition'' of $2g-2$, the same power series from the reduction of Conjecture \ref{conj:upperbound} for $\ell = 1$ generalizes directly to strata of $\ell$-differentials for all $\ell$. See Remark~\ref{rem:conj-ell} for more details.

 \smallskip
Combining Theorems \ref{thm:lowerbound} and \ref{thm:upperbound} shows that $R^*(\P^1(\mu)) = \qq[\eta]/(\eta^{\lfloor g/3 \rfloor + 1})$ for all positive partitions with at least $4g/3$ simple zeros whenever $g \leq 30$. In these cases, our proof demonstrates that all relations in $R^*(\P^1(\mu))$ are inherited from \eqref{iota} and relations in $R^*(\M_g)$.

\medskip
Theorem \ref{thm:lowerbound} is a non-vanishing result for tautological rings of strata of differentials with many simple zeros. We also have the following result for strata associated to partitions whose odd entries occur in pairs.

\begin{thm}
\label{thm:varying}
Let $\mu$ be a zero type for holomorphic abelian differentials. For sufficiently large $g$, if the odd entries of $\mu$ appear in pairs and if the number of these odd pairs is at least four, then $\P^1(\mu)$ is a varying stratum, and consequently $\eta \neq 0$ in this case. 
\end{thm}

The above condition on $\mu$ ensures that $\P^1(\mu)$ contains hyperelliptic differentials. See Proposition~\ref{prop:varying} for a precise expression of such $\mu$. Here, a stratum of holomorphic abelian differentials is said to be {\em varying} if it contains two Teichm\"uller curves with different area Siegel--Veech constants. See \cite{CMAbelian} for an introduction to these concepts. Since the area Siegel--Veech constant determines the ratio of the intersection numbers of a Teichm\"uller curve with $\eta$ and the boundary divisor of the stratum, Theorem~\ref{thm:varying} then follows from showing that area Siegel--Veech constants of Teichm\"uller curves in the locus of hyperelliptic differentials differ from the large genus limit of area Siegel--Veech constants of generic differentials. 

Finally, we remark that the orbifold fundamental group of strata of differentials with sufficiently many simple zeros can be understood from a similar perspective. Moreover, it is also meaningful to study the cohomology of compactified strata and find non-tautological classes. We refer to \cite{Qiu25, Salter25, CDGM26} for recent developments regarding these questions. 

\subsection*{Acknowledgements} This work was initiated when the first named author visited Berkeley in December 2024. He thanks David Eisenbud for the invitation and hospitality. We thank Philip Tosteson for sharing and discussing his work with us. The research of the first named author was partially supported by the National Science Foundation under Grant DMS-2301030, Simons Travel Support for Mathematicians, and a Simons Fellowship under Record ID SFI-MPS-SFM-00005694. This research was partially conducted during the period the second named author served as a Clay Research Fellow.

\section{The partial compactification $\overline{X}^\ell(m_1, \ldots, m_k)$} \label{sec:lower}

The key geometric construction behind our proof of Theorem \ref{thm:lowerbound} is to relate $\P^\ell (\mu)$ to a projective bundle over $\M_{g,k}$, by forgetting a subset of the simple zeros.
Suppose that $\mu = (m_1, \ldots, m_n)$ and $m_i = 1$ for $i \geq k + 1$.
Consider the forgetful map
\[\P^\ell(m_1, \ldots, m_k, 1^{n - k})/\mathbb{S}_{n - k} \to \M_{g,k},\]
which forgets the last $n - k$ zeros, all of which are simple.
As we will see, away from loci of high codimension, this map factors through a projective bundle over an open substack of $\M_{g,k}$. This allows us to leverage the following theorem of Harer and Madsen--Weiss about the stable
cohomology of $\M_{g,k}$. The following ranges for stability can be found in \cite{Wahl}.

\begin{thm} \label{stab}
We have
\[\qq[\kappa_1, \kappa_2, \ldots, \psi_1, \ldots, \psi_k] \to H^{*}(\M_{g,k})\]
is injective for $* \leq \frac{2}{3}g$ and surjective for $* \leq \frac{2}{3}g - \frac{2}{3}$.
\end{thm}

To translate from \cite{Wahl} to the above result, let $\Gamma_{g,r}^k$ be the mapping class group of a genus $g$ surface with $k$ punctures and $r$ boundary components. Then $H_*(\M_{g,k}) = H_*(\Gamma_{g,0}^k)$. By the universal coefficient theorem, this is dual to $H^*(\M_{g,k})$.
Consider the commutative diagram
\begin{center}
\begin{tikzcd}
H_*(\Gamma_{g,1}^k) \arrow{d} \arrow{r} & H_*(\Gamma_{g+1,1}^k) \arrow{d} \\
H_*(\Gamma_{g,0}^k) \arrow{r} & H_*(\Gamma_{g+1,0}^k).
\end{tikzcd}
\end{center}
For $* \leq \frac{2}{3}g$, the top arrow is surjective by \cite[Theorem 1.1]{Wahl} and the vertical arrows are isomorphisms by \cite[Theorem 1.2]{Wahl}, so the bottom arrow is surjective. (As explained at the bottom of \cite[p.\ 2]{Wahl}, Theorems 1.1 and 1.2 also hold with punctures.)
Dualizing, the map from stable cohomology is injective in degrees $* \leq \frac{2}{3}g$. Meanwhile, if $* \leq \frac{2}{3}g - \frac{2}{3}$, then \cite[Theorems 1.1 and 1.2]{Wahl} shows that the top horizontal and vertical maps above are isomorphisms. Therefore, the bottom map is also an isomorphism, so all cohomology is stable for $* \leq \frac{2}{3}g - \frac{2}{3}$.

\subsection{Notation and Borel--Moore homology}
Recall that if $X$ is a smooth variety of dimension $n$, its ordinary cohomology $H^i(X)$ is isomorphic to Borel--Moore homology in complementary degree
$H^{\mathrm{BM}}_{2n - i}(X)$ by 
Poincar\'e duality.
Starting now, given $X$ of dimension $n$, possibly singular, when we write $H^i(X)$, we shall actually mean $H^{\mathrm{BM}}_{2n - i}(X)$. In other words, we are talking about Borel--Moore homology groups, indexed by codimension.
Thus, for example, given a proper morphism $X \to Y$ with $\dim Y - \dim X = c$, we have pushforward maps $H^{i - 2c}(X) \to H^i(Y)$.
Moreover, if $X \to Y$ is proper and surjective, then this pushforward map is surjective on the lowest weight part
\[W_{-k}H_k^{\mathrm{BM}}(X) \twoheadrightarrow W_{-k}H^{\mathrm{BM}}(Y),\]
which we instead write as
\[W_{i-2c}H^{i-2c}(X) \twoheadrightarrow W_i H^i(Y).\]
If $Z \subset X$ is a closed subvariety of codimension $c$ and $U = X \smallsetminus Z$ is its open complement, then the \emph{localization sequence} in Borel--Moore homology
is the long exact sequence
\[\cdots \to H_j^{\mathrm{BM}}(Z) \to H_j^{\mathrm{BM}}(X) \to H_j^{\mathrm{BM}}(U) \to H_{j-1}^{\mathrm{BM}}(Z) \to \cdots \]
We shall write this sequence as
\[\cdots H^{i - 2c}(Z) \to H^i(X) \to H^i(U) \to H^{i-2c+1}(Z) \to \cdots,\]
even when the spaces involved are not necessarily smooth. In particular, we have isomorphisms
$H^i(X) \cong H^i(U)$ for $i \leq 2c - 2$.

\subsection{Relating $\P^\ell(\mu)$ to a projective bundle} \label{sec:rel}
Consider the map 
\[\P^\ell(m_1, \ldots, m_k, 1^{n - k}) \to \M_{g,k}\] defined by forgetting the last $n - k$ zeros, all of which are simple. 
Let $r$ be such that $m_1, \ldots, m_r < 0$ and $m_{r+1}, \ldots, m_k \geq 0$.
We will define and study some closely related spaces over $\M_{g,k}$.
First, we define a space $\overline{X}^\ell(m_1, \ldots, m_k)  \to \M_{g,k}$ whose fiber over a point $(C, p_1, \ldots, p_k)$ is $\pp H^0(C, K^{\otimes \ell}(- m_1p_1 - \cdots - m_k p_k))$. For $m_i < 0$, this twist allows a pole of order $-m_i$ at $p_i$, and for $m_i \geq 0$,  this twist forces a zero of order $m_i$ at $p_i$.

To construct $\overline{X}^{\ell}(m_1, \ldots, m_k)$, let $\pi \colon \C \to \M_{g,k}$ be the universal curve, let $\omega$ be the relative canonical bundle, and let $D_1, \ldots , D_k$ denote the images of the $k$ sections corresponding to the marked points. Define the line bundle
\[\K^\ell_\mu \colonequals \omega^{\otimes \ell}(-m_1 D_1 - \cdots - m_rD_r). \]
On $\C$, we have an exact sequence of sheaves
\[0 \rightarrow \K^\ell_\mu(-m_{r+1}D_{r+1} \cdots - m_kD_k) \rightarrow \K^\ell_\mu \rightarrow \bigoplus_{i=r+1}^k \left. \K^\ell_\mu\right|_{m_i D_i} \rightarrow 0.\]
Now consider the projective bundle $h\colon  \pp\left(\pi_*\K^\ell_\mu\right) \to \M_{g,k}$. 
Define $\overline{X}^\ell(m_1, \ldots, m_k) \subset \pp\left(\pi_*\K^\ell_\mu\right)$ as the vanishing locus of the composition
\[\O_{\pp\left(\pi_*\K^\ell_\mu\right) }(-1) \to h^* \pi_*\K^\ell_\mu \to \bigoplus_{i=r+1}^k \left. h^*\pi_* \K^\ell_\mu \right|_{m_i D_i}.\]
In particular,
\[\phi\colon \overline{X}^\ell(m_1,\ldots, m_k) \to \M_{g,k}\]
is proper. We can think of $\overline{X}^\ell(m_1, \ldots, m_k)$ as the moduli space of pluricanonical divisors on $k$-pointed curves such that the only negative coefficients in the divisor are at the first $r$ marked points and the coefficient of $p_i$ is at least $m_i$ for each marked point $p_i$.
If $k = 0$, then $\overline{X}^\ell(\varnothing)$ is simply equal to $\pp(\pi_* \omega^{\otimes \ell})$, the projectivized $\ell$-Hodge bundle.

Next, we define $X^\ell(m_1, \ldots, m_k) \subset \overline{X}^\ell(m_1, \ldots, m_k)$ to be the open subset where the multiplicities of marked points in the pluricanonical divisor are exactly $m_i$ and all other points occur with multiplicity $1$. Then, by construction
\begin{equation} \label{quotient} \P^\ell(m_1, \ldots, m_k, 1^{n - k})/\mathbb{S}_{n - k} \cong X^\ell(m_1, \ldots, m_k)
\end{equation}
and the map $\P^\ell(m_1, \ldots, m_k, 1^{n - k}) \to \M_{g,k}$ factors through $X^\ell(m_1, \ldots, m_k)$.

\subsubsection{The case $\ell = 1$}
Let $m = m_1 + \cdots + m_k$ and suppose $m < g + \delta_{0r} - 1$, where $\delta_{0r} = 1$ if $r = 0$ and $\delta_{0r} = 0$ otherwise.
Note that if $r = 0$, then 
\[h^1(C, K(-m_1p_1 - \cdots - m_kp_k)) \geq 1\]
for all $(C, p_1, \ldots, p_k) \in \M_{g,k}$. However, if $r > 0$, then, since we assume $m < g$, we have
\[h^1(C, K(-m_1p_1 - \cdots - m_kp_k)) = 0\] 
for general 
$(C, p_1, \ldots, p_k) \in \M_{g,k}$.
Accordingly, we define
 $Z^1(m_1, \ldots, m_k) \subset \M_{g,k}$ by the condition
\begin{align*} Z^1(m_1, \ldots, m_k) &\colonequals \{(C, p_1, \ldots, p_k) : h^1(C, K( - m_1p_1 - \cdots - m_kp_k)) > \delta_{0r}
\} \\
&= \{(C, p_1, \ldots, p_k) : h^0(C,\O(m_1p_1 + \cdots + m_kp_k)) > \delta_{0r}
\}.
\end{align*}

\begin{lem}
\label{lem:exceptional}
Assume $m < g + \delta_{0r} -1$.
Let $Z = Z^1(m_1, \ldots, m_k)$. Then
\[\codim(Z \subset \M_{g,k}) > \codim(\phi^{-1}(Z) \subset \overline{X}^1(m_1 ,\ldots, m_k)) \geq g - m + \delta_{0r} -1.\]  
\end{lem}

\begin{proof}
The condition $m < g + \delta_{0r} - 1$ ensures that 
$\phi$ is dominant.
The first inequality then follows from the fact that the fiber dimension of $\phi$ over $Z$ is strictly greater than the general fiber dimension.

For the second inequality, let $Z_v(m_1,\ldots, m_k) \subset \M_{g,k}$ be the locally closed subset of $(C, p_1, \ldots, p_k)$ such that \[h^0(C, \O(m_1p_1 + \cdots + m_kp_k)) = v + \delta_{0r},\]
so that $Z^1(m_1, \ldots, m_k) = \bigcup_{v \geq 1} Z_v(m_1,\ldots, m_k)$.
We will show that
\begin{equation} \label{need} \dim Z_v(m_1, \ldots, m_k) \leq 2g - 2 - \delta_{0r} + k + m - v.
\end{equation}
This will establish that the codimension of $Z_v(m_1, \ldots, m_k)$ in $\M_{g,k}$ is at least
\[(3g - 3 + k) - (2g - 2 - \delta_{0r} + m + k - v) = g - m + v-1 + \delta_{0r}.\]
Moreover, the fiber dimension of $\phi$ over $Z_v(m_1, \ldots, m_k)$
is $v$ more than the generic fiber dimension, so this will also show that the codimension of $\phi^{-1}(Z)$ is at least $g - m + \delta_{0r} - 1$. 

Observe that for any collection of $(a_1, \ldots, a_k)$ with $0 \leq a_i \leq m_i$ and $a_1 + \cdots + a_k = a$, we have 
\[h^0(\O((m_1-a_1)p_1 + \cdots + (m_k - a_k) p_k) \geq h^0(\O(m_1p_1 + \cdots + m_kp_k) ) - a.\]
Hence, $Z_{v}(m_1, \ldots, m_k) \subseteq Z_{v-a}(m_1-a_1, \ldots, m_k - a_k)$.
Setting $a = v - 1$, we find in particular that 
$$\dim Z_{v}(m_1, \ldots, m_k) \leq \dim Z_{1}(m_1-a_1, \ldots, m_k - a_k).$$ 
Thus, it suffices to prove \eqref{need} in the case $v = 1$.
Note additionally that the forgetful map $Z_1(m_1, \ldots, m_k) \to Z_1(\{m_i: m_i \neq 0\})$ has fibers of dimension $\#\{m_i : m_i = 0\}$, so it suffices to prove \eqref{need} in the case that all $m_i \neq 0$.

First suppose $r = 0$. By the last sentence of the previous paragraph, we may assume $m_i > 0$ for all $i$.
We must bound the dimension of the locus of $(C, p_1, \ldots, p_k)$ with $h^0(C, \O(m_1p_1 + \cdots + m_k p_k)) = 2$. 
We claim it suffices to bound the dimension of the locus of 
$(C, p_1, \ldots, p_k)$
for which  $\O(m_1p_1 + \cdots + m_k p_k)$ is base point free.  Indeed, if this line bundle has a base point, then
subtracting the base point and 
inducting on $m$ and $k$, we achieve the desired bound.
Thus, we can assume that the divisor $m_1p_1 + \cdots + m_kp_k$ is the fiber of a degree $m$ cover $C \to \pp^1$.
Consider the Hurwitz space of branched covers of $\mathbb P^1$ with degree $m$ having a fiber of ramification type $(m_1, \ldots, m_k)$. 
This fiber contributes $m - k$ to the total ramification degree.
By Riemann--Hurwitz, in addition to the image point of the special fiber, there are at most $2g-2 + 2m - (m - k)$ other branch points. Hence, the dimension of the Hurwitz space of such covers is at most
\[2g - 2 + 2m - (m - k) + 1 - 3 = 2g - 3 + m + k - 1.\]
This is an upper bound for $\dim Z_1(m_1, \ldots, m_k)$, thus
proving \eqref{need}.

Next we treat the case $r > 0$. 
Now we must bound the dimension of the locus of $(C, p_1, \ldots, p_k) \in \M_{g,k}$ such that $h^0(C, \O(m_1p_1 + \cdots + m_k p_k)) = 1$. This means there exists some effective divisor $E$ such that
\[\sum_{i \leq r} -m_i p_i + E \sim \sum_{i > r} m_i p_i.\]
If $p_i$ for $i > r$ is contained in $E$, then we can subtract it from both sides. As before, inducting on $m$ and $k$, it suffices to treat the case when $E$ is disjoint from $p_i$ for $i > r$.
This now implies that there exists a branched cover $C \to \pp^1$ whose fiber over $0$ is $\sum_{i > r} m_i p_i$ and whose fiber over $\infty$ is $\sum_{i \leq r} -m_i p_i + E$. Consider the Hurwitz space of covers of $\pp^1$ of degree $d = m_{r+1} + \cdots + m_k$
with two special fibers: one of ramification profile $(m_{r+1}, \ldots, m_k)$ and one having ramification profile a specialization of $(-m_1, \ldots, -m_r,1, \ldots, 1)$. The first special fiber contributes $d - (k-r)$ to the total ramification degree. The second special fiber contributes at least $-m_1 - \cdots - m_r - r$ to the total ramification degree. By Riemann--Hurwitz, there are at most
\[2g - 2 + 2d - (d - (k-r)) - (-m_1 - \cdots - m_r - r) = 2g - 2 + m + k\]
other branch points. Accounting for the two special branch points and the automorphisms of $\pp^1$, 
the dimension of the Hurwitz space of such covers is at most $2g - 3 + m + k$,
which gives the desired upper bound in \eqref{need}.
\end{proof}

\subsubsection{The case $\ell \geq 2$}
Given $m_1, \ldots, m_k$, let $m = m_1 + \cdots + m_k$. Additionally, assume that $m < \ell(2g - 2) - g + 1$. 
We define $Z^\ell(m_1, \ldots, m_k) \subset \M_{g,k}$ for $\ell \geq 2$ by the condition
{\small \begin{align*} Z^\ell(m_1, \ldots, m_k) &\colonequals \{(C, p_1, \ldots, p_k) : h^1(C, K^{\otimes \ell}( - m_1p_1 - \cdots - m_kp_k)) > 
0
\} \\
&= \{(C, p_1, \ldots, p_k) : h^0(C,K^{\otimes (1 - \ell)}(m_1p_1 + \cdots + m_kp_k)) > 0
\}.
\end{align*}}

\noindent
In other words, $Z^\ell(m_1, \ldots, m_k) \subset \M_{g,k}$ is the locus over which the fiber dimension of \[\phi: \overline{X}^\ell(m_1, \ldots, m_k) \to \M_{g,k}\] jumps up.

\begin{lem} \label{l>1} Assume $m < \ell(2g - 2) - g + 1$. 
Let $Z = Z^\ell(m_1, \ldots, m_k)$. Then we have
\begin{align} \codim(Z \subset \M_{g,k}) > \codim(\phi^{-1}(Z) \subset \overline{X}^\ell(m_1, \ldots, m_k)) &\geq (2\ell - 1)(2g - 2) - 2m-1 \notag \\
&\geq \ell(2g - 2) - g - m + 1. \notag
\end{align}
\end{lem}

\begin{proof}
The assumption $m < \ell(2g - 2) - g + 1$ implies that $\phi$ is dominant.
The first inequality then follows from the fact that the fibers of $\phi$ over $Z$ have dimension strictly greater than those of the general fibers of $\phi$.

To prove the second inequality, notice that a point of $\phi^{-1}(Z)$ corresponds to 
a divisor of type 
    $$m_1p_1 + \cdots + m_k p_k - \sum_{j=1}^{m-(\ell-1)(2g-2)} q_j\sim (\ell-1)K,$$ 
    where the $q_j$ can merge with each other or merge with the $p_i$ (which however can only make the dimension of the corresponding stratum of $(\ell-1)$-differentials decrease). 
    It follows that 
    \begin{align} \dim \phi^{-1}(Z) & \leq \dim \P^{\ell-1}(m_1, \ldots, m_k, 1^{m - (\ell-1)(2g-2)}) \notag \\ 
    & \leq 2g-2 + k + m - (\ell-1)(2g-2). \label{dpz} 
    \end{align}
    (Here we assume $m \geq (\ell-1)(2g-2)$ since if $m < (\ell-1)(2g-2)$, then $Z$ is empty.)
Meanwhile, we have 
\begin{equation} \label{dx} \dim \overline{X}^\ell(m_1, \ldots, m_k) = \dim \P^\ell(m_1, \ldots, m_k, 1^{\ell(2g - 2)-m}) = 2g-3+k + \ell(2g-2) -m.
\end{equation}
Subtracting \eqref{dpz} from \eqref{dx} yields the desired lower bound on the codimension of $\phi^{-1}(Z)$.

The third inequality is a rearrangement of the assumption $m < \ell(2g - 2) - g + 1$.
\end{proof}

\subsection{The stable cohomology of $\overline{X}^{\ell}(m_1, \ldots, m_k)$} \label{Xbar}
Given $\ell$ and $m_1, \ldots, m_k$,
let $U \subset \M_{g,k}$ be the open complement of $Z^\ell(m_1, \ldots, m_k)$.
Over $U$, cohomology and base change holds and $\pi_*\omega^{\otimes \ell}(-m_1D_1 - \cdots - m_kD_k)$ is a vector bundle. We have
\begin{equation} \label{rk} \rank \pi_*\omega^{\otimes \ell}(-m_1D_1 - \cdots - m_kD_k)|_U = 
\ell(2g - 2) - g + 1 - m + \delta_{0r} \delta_{1\ell}.
\end{equation}
We define
\begin{align} \label{Ydef} \overline{Y}^\ell(m_1, \ldots, m_k) &\colonequals
\overline{X}^\ell(m_1, \ldots, m_k) \smallsetminus \phi^{-1}(Z^\ell(m_1, \ldots, m_k))\\
\label{Y2}  &= \pp(\pi_*\omega(-m_1D_1 - \cdots - m_kD_k)|_U).
\end{align}

The lemmas of the previous section show that $\overline{X}^\ell(m_1,\ldots,m_k)$ is ``well-approximated" by the projective bundle $\overline{Y}^\ell(m_1,\ldots,m_k)$. Combining Lemmas \ref{lem:exceptional} and \ref{l>1} with our knowledge of the cohomology of $\M_{g,k}$ in Theorem \ref{stab}, we can use this to determine the stable cohomology of $\overline{X}^\ell(m_1,\ldots,m_k)$.
We define
\begin{align*} \mathsf{i}^\ell_r(m) &\colonequals
\min\{\tfrac{2}{3}g, 2(\ell(2g - 2) - g - m + \delta_{0r}\delta_{1\ell})\} \\
\mathsf{s}_r^\ell(m) &\colonequals
\min\{\tfrac{2}{3}g - \tfrac{2}{3}, 2(\ell(2g - 2) - g - m + \delta_{0r}\delta_{1\ell})\}.
\end{align*}

\begin{lem} \label{lem:stab}
Let $m = m_1 + \cdots + m_k$, and let $r$ be such that $m_i < 0$ for $i \leq r$ and $m_i \geq 0$ for $i \geq r+1$.
The map
\begin{equation} \label{tauts} \qq[\kappa_1,\kappa_2,\ldots, \psi_1,\ldots,\psi_k,\eta] \to H^*(\overline{X}^\ell(m_1,\ldots,m_k))\end{equation}
is injective in degrees $* \leq \mathsf{i}^\ell_r(m)$ and
surjective in degrees $* \leq \mathsf{s}_r^\ell(m)$. 
\end{lem}

\begin{proof}
Let $Z = Z^\ell(m_1, \ldots,m_k)$ and let $U = \M_{g,k} \smallsetminus Z$.
Our lower bounds on the codimension of $\phi^{-1}(Z) \subset \overline{X}^\ell(m_1, \ldots, m_k)$ in Lemmas \ref{lem:exceptional}
and \ref{l>1} are strictly greater than 
\[\ell(2g - 2) - g - m + \delta_{0r}\delta_{1 \ell}.\]
Thus, the localization sequence associated with Equation \eqref{Ydef} implies that 
\[H^*(\overline{X}^\ell(m_1, \ldots, m_k)) = H^*(\overline{Y}^\ell(m_1, \ldots, m_k))\]
for $* \leq \mathsf{i}^\ell_r(m)$.
Similarly, localization also shows that $H^*(U) = H^*(\M_{g,k})$ for $* \leq \mathsf{i}^\ell_r(m)$.

Meanwhile \eqref{Y2}
shows that the cohomology of $\overline{Y}^\ell(m_1, \ldots, m_k)$ is  generated over $H^*(U)$ by $\eta$ with no relations in degrees less than or equal to twice the rank in \eqref{rk} minus $2$, which is the second number in the definition of $\mathsf{i}^\ell_r(m)$.
The assumption $* \leq \frac{2}{3}g$ ensures that there are no relations among tautological classes on $H^*(\M_{g,k})$, so \eqref{tauts} is injective for $* \leq \mathsf{i}^\ell_r(m)$. Meanwhile, if $* \leq \frac{2}{3}g - \frac{2}{3}$, then the tautological classes generate $H^*(U)$ so \eqref{tauts} is surjective.
\end{proof}

\section{The pure weight cohomology of $X^\ell(m_1, \ldots, m_k)$} \label{sec:pure}
Let $W_*H^*(X) = \bigoplus_{i} W_{i}H^i(X)$ denote the subring of pure weight cohomology of $X$. In this section, we prove the following result.

\begin{thm} \label{purewt}
Let $\mu = (m_1, \ldots, m_k, 1^{\ell(2g - 2) - m})$ with $m = m_1 + \cdots + m_k$
the sum of $k$ specified orders such that the remaining orders are all equal to $1$.
Let $r$ be such that $m_i < 0$ for $i \leq r$ and $m_i \geq 0$ for $i > r$.

If $m_i \neq -\ell$ for all $i$, then
\begin{equation} \label{t11}
\qq[\eta] \rightarrow W_*H^*(\P^\ell(m_1, \ldots, m_k,1^{\ell(2g-2)-m})/\mathbb{S}_{\ell(2g-2)-m})
\end{equation}
is injective in degrees $* \leq \min\{\mathsf{i}^\ell_r(m), \mathsf{s}^\ell_r(m+2)+2\}$ and surjective in degrees $* \leq \mathsf{s}^\ell_r(m)$.

If some $m_i = -\ell$, let $\{i_1, \ldots, i_s\} \subset \{1, \ldots, n\}$ be the  indices for which $m_{i_j} = -\ell$. Then
\begin{equation} \label{t22}
\qq[\psi_{i_1}, \ldots, \psi_{i_s}] \rightarrow  W_*H^*(\P^\ell(m_1, \ldots, m_k,1^{\ell(2g-2)-m})/\mathbb{S}_{\ell(2g-2)-m})
\end{equation}
is injective in degrees $* \leq \min\{\mathsf{i}^\ell_r(m), \mathsf{s}^\ell_r(m+2)+2\}$ and surjective in degrees $* \leq \mathsf{s}^\ell_r(m)$.
\end{thm}

One readily checks that $* \leq \min\{\mathsf{i}^\ell_r(m), \mathsf{s}^\ell_r(m+2)+2\}$ is the range given for injectivity in the statement of Theorem \ref{thm:lowerbound}.
Since \eqref{t11} and \eqref{t22} factor through
\eqref{t1} and \eqref{t2} respectively, Theorem \ref{purewt} implies Theorem \ref{thm:lowerbound}.

\begin{proof}[Proof of Theorem \ref{purewt}]
Recall that we write
\[X^\ell(m_1, \ldots, m_k) = \P^\ell(m_1, \ldots, m_k,1^{\ell(2g-2)-m})/\mathbb{S}_{\ell(2g-2)-m}\]
and our partial compactification
$\overline{X}^\ell(m_1,\ldots,m_k)$ is the moduli space of pluricanonical divisors on $k$-pointed curves so that the coefficient of the $i$th marked point is at least $m_i$ and coefficients of unmarked points are nonnegative. Throughout this section, we fix $\ell$ and $m_1, \ldots,m_k$ and write
\[\overline{X} = \overline{X}^\ell(m_1, \ldots, m_k) \qquad \text{and} \qquad X = X^\ell(m_1, \ldots, m_k).\]

For $i = 1, \ldots, k$, define $A_i \subset \overline{X}$ to be the locus of divisors where the coefficient of the $i$th marked point is strictly greater than $m_i$. 
We have 
\[A_i \cong \overline{X}^\ell(m_1, \ldots, m_i + 1, \ldots, m_k) \]
except in the following two cases:
\begin{itemize}
\item If $\ell = r = 1$ and $m_1 = -2$, then $A_1 \cong \overline{X}^\ell(0,m_2, \ldots, m_k)$.
\item If $\ell = 1$ and $r = 2$ and $m_1 = m_2 = -1$, then $A_1 = A_2 \cong \overline{X}^\ell(0,0,m_3, \ldots, m_k)$.
\end{itemize}
Both exceptional cases are due to the fact that, by the residue theorem,  there does not exist a  meromorphic abelian differential on a compact Riemann surface with a unique pole which is simple.
Let $A = A_1 \cup \cdots \cup A_k$ be the union of these divisors.
Next, we define $B \subset \overline{X} \smallsetminus A$ to be the divisor where the support of the pluricanonical divisor has multiplicity $2$ or more at an unmarked point.
Then we have
\[X = \overline{X} \smallsetminus (A \cup B).\]

The boundary divisors $A_i$ are linear in each fiber of $\overline{X} \to \M_{g,k}$, making them relatively easier to understand. In fact, as we have seen, away from some loci of high codimension, $\overline{X} \to \M_{g,k}$ is a projective bundle, and $A_i$ is a projective subundle. To find the class of $A_i$, we describe it as the vanishing locus of a section of a line bundle. Let $\pi\colon \C \to \overline{X}$ be the pullback of the universal curve over $\M_{g,k}$ and let $\sigma_1,\ldots, \sigma_k\colon \overline{X} \to \C$ be the disjoint sections corresponding to the marked points. We write $D_i \subset \C$ for the image of $\sigma_i$.

Recall that $\overline{X}$ is a closed substack of a projective bundle over $\M_{g,k}$. Let $\O(-1)$ denote the restriction of the relative $\O(-1)$ for this projective bundle to $\overline{X}$, and let $\eta = c_1(\O(-1)) \in A^1(\overline{X})$. Then the pullback of this class $\eta$ along
\[\P^\ell(m_1, \ldots, m_k,1^{\ell(2g-2)}) \to X \to \overline{X}\]
is the class we called $\eta$ in the introduction.
On $\C$, there is a tautological morphism of line bundles
\begin{equation}  \label{ts} \pi^*\O(-1) \to \L \colonequals \omega^{\otimes \ell}\left(-\sum_{i=1}^k m_i D_i\right)\end{equation}
which on each fiber of $\pi$ corresponds to the meromorphic section of $\omega^{\ell}$ whose locus of zeros and poles is the divisor specified by the point in $\overline{X}$. This morphism vanishes along $D_i$ precisely when the $i$th marked point has multiplicity greater than $m_i$. Thus $A_i$ is the vanishing locus of the map of line bundles
\[\O(-1) = \sigma_i^*\pi^*\O(-1) \to \sigma_i^*\L.\]
Consequently, we have
\begin{equation} \label{Aclass} [A_i] = c_1(\O(1) \otimes \sigma_i^*\L) = -\eta + (\ell + m_i) \psi_i \in A^1(\overline{X}).
\end{equation}
Excising $A_i$ from $\overline{X}$, we recover the first relation in \eqref{iota} from the introduction.
Moreover, when $*-2 \leq \mathsf{s}^\ell_r(m+2)$, Lemma \ref{lem:stab} implies that $H^{*-2}(A_i)$ is generated by the tautological classes $\eta, \kappa_1, \kappa_2, \ldots, \psi_1, \ldots, \psi_k$. As all of these classes are pulled from $\overline{X}$, the push-pull formula tells us that the image of
\[W_{*-2}H^{*-2}(A_i) \to W_*H^*(\overline{X})\]
agrees with the ideal generated by $[A_i]$ when $* -2 \leq \mathsf{s}^\ell_r(m+2)$. 
Additionally, when $* \leq \mathsf{i}^\ell_r(m)$, there are no relations among the tautological classes on $W_*H^*(\overline{X})$.
Thus, the kernel of 
\[\qq[\eta, \kappa_1, \kappa_2, \ldots, \psi_1, \ldots, \psi_k] \to W_*H^*(\overline{X}) \to W_*H^*(\overline{X} \smallsetminus A)\]
is the ideal generated by the $[A_i]$.

Meanwhile, the closure of $B$ is nonlinear: locally in each fiber of $\overline{X} \to \M_{g,k}$ it looks like the discriminant locus of polynomials with a double root. Thus, it will be useful to construct a linearization $\tilde{B}$ which marks the double root. 
The points of $\C$ correspond to $(k+1)$-pointed curves $(C, H, p_1, \ldots, p_k, z_{k+1})$ where the $p_i$ are distinct, but $z_{k+1}$ is unconstrained, together with a pluricanonical divisor $H$ such that $H - \sum_{i=1}^k m_i p_i$ is effective. We define $\tilde{B} \subset \C$ to be the locus of $(C, H, p_1, \ldots, p_k, z_{k+1})$ such that $H - \sum_{i=1}^k m_ip_i - 2z_{k+1}$ is effective.
In other words, $\tilde{B} \subset \C$ is defined by the vanishing of the first principal parts of the map in \eqref{ts}. 
The bundle of first principal parts of $\pi^*\O(1) \otimes \L$ is filtered by
\[0 \rightarrow \pi^*\O(1) \otimes \L \otimes \omega \rightarrow P^1(\pi^*\O(1) \otimes \L) \rightarrow \pi^*\O(1) \otimes \L\rightarrow 0.\]
Consequently,
we have
\begin{align*} [\tilde{B}] &= c_2(P^1(\pi^*\O(1) \otimes \L)) = c_1(\pi^*\O(1) \otimes \L) \cdot c_1(\pi^*\O(1) \otimes \L \otimes \omega) \\
&= \left(-\eta + \ell \psi_{k+1} - \sum_{i=1}^k m_i D_i\right)\left(-\eta + (\ell+1) \psi_{k+1} - \sum_{i=1}^k m_i D_i\right) \\
&= \ell(\ell+1) \psi_{k+1}^2 + \cdots \in H^4(\C),
\end{align*}
where $\psi_{k+1} = c_1(\omega) \in H^2(\C)$. Notice in particular that
\begin{equation} \label{notice} \pi_*(\psi_{k+1}^i\cdot [\tilde{B}]) = \ell(\ell+1) \kappa_{i+1} + \cdots \in H^{2i+2}(\overline{X}),
\end{equation}
where the other terms involve only the $\kappa_j$ with $j < i+1$.

Since $\tilde{B} \subset \C$ is closed, $\tilde{B} \to \overline{X}$ is proper. Thus,
\begin{equation} \label{bt} \tilde{B} \smallsetminus \tilde{B} \cap \pi^{-1}(A) \to \overline{X} \smallsetminus A 
\end{equation}
is also proper. Moreover, the image of \eqref{bt} consists of those divisors that have a point of multiplicity $2$ that is distinct from the marked points, which is precisely
$B$. Since the pushforward along proper surjective maps is surjective on pure weight cohomology, the top row of Figure 1 below is exact.

\begin{figure}[h]
\label{fig:summary}
\hspace*{-0.7cm}
\begin{tikzcd}
W_{*-2}H^{*-2}(\tilde{B} \smallsetminus \tilde{B} \cap \pi^{-1}(A)) \arrow{r} & W_{*}H^{*}(\overline{X} \smallsetminus A) \arrow{r} & W_*H^*(X) \\
 W_{*-2} H^{*-2}(\tilde{B}) \arrow[two heads]{u} \arrow{r} & W_*H^*(\overline{X}) \arrow[two heads]{u} \\
 \qq[\eta, \kappa_1, \kappa_2,\ldots, \psi_1,\ldots, \psi_k,\psi_{k+1}] \arrow[bend left = 90, color=blue, two heads]{uu} \arrow{u} \arrow[bend right = 10, color=red]{ur}{\pi_*(- \cdot [\tilde{B}])}
 & \bigoplus_{i=1}^k W_{*-2}H^{*-2}(A_i) \arrow{u} \\
 & \bigoplus_{i=1}^k \qq[\eta, \kappa_1, \kappa_2, \ldots, \psi_1,\ldots, \psi_k] \arrow[two heads, color=blue]{u} \arrow[color=red, bend right = 90, swap]{uu}{\sum (-)_i \cdot [A_i]}
\end{tikzcd}
\caption{
The top row and top three terms in the right column
are exact.
The blue arrows are both surjective when  $* - 2 \leq \mathsf{s}^\ell_r(m+2)$
by Lemma
\ref{lem:stab} and \eqref{open}. Consequently, the kernel of $W_*H^*(\overline{X}) \to W_*H^*(X)$ is generated by the images of the two red maps to it.}
\end{figure}

Next, we claim that
\begin{equation} \label{inc1} \tilde{B} \smallsetminus \tilde{B} \cap \pi^{-1}(A) \subset \C \smallsetminus (D_1 \cup \cdots \cup D_k).
\end{equation}
Indeed, removing $\pi^{-1}(A)$ removes all divisors that have multiplicity greater than $m_i$ at $p_i$. Since $\tilde{B}$ parameterizes $(C, H, p_1, \ldots, p_k)$ with $H = \sum_{i=1}^k m_i p_i + 2z_{k+1} + \cdots$, we see that $z_{k+1} \neq p_i$ on the complement of $\pi^{-1}(A)$. Moreover the condition that $z_{k+1} \neq p_i$ shows that there is an open embedding
\begin{equation} \tilde{B} \smallsetminus \tilde{B} \cap \pi^{-1}(A) \subset \overline{X}^\ell(m_1, \ldots, m_k, 2). 
\end{equation}
In particular, we have a surjection
\begin{equation} \label{open} W_{*-2}H^{*-2}(\overline{X}^\ell(m_1, \ldots, m_2,2)) \to W_{*-2}H^{*-2}( \tilde{B} \smallsetminus \tilde{B} \cap \pi^{-1}(A)). 
\end{equation}

By the push-pull formula, all morphisms in the diagram are $\qq[\eta, \kappa_1, \kappa_2, \ldots, \psi_1, \ldots, \psi_k]$-module homomorphisms. By the blue surjective arrow on the left of Figure 1, we have that $W_{*-2}H^{*-2}(\tilde{B} \smallsetminus \tilde{B} \cap \pi^{-1}(A))$ is generated as a module over $\qq[\eta, \kappa_1, \kappa_2, \ldots, \psi_1, \ldots, \psi_k]$ by the classes $1, \psi_{k+1}, \psi_{k+1}^2, \ldots$. 
Hence, the image of 
\begin{equation} \label{btpush} W_{*-2}H^{*-2}(\tilde{B} \smallsetminus \tilde{B} \cap \pi^{-1}(A)) \to W_*H^*(\overline{X} \smallsetminus A)
\end{equation}
is generated as a module over 
$\qq[\eta, \kappa_1, \kappa_2,\ldots, \psi_1, \ldots, \psi_k]$ by the pushforwards of the classes $1, \psi_{k+1}, \psi_{k+1}^2, \ldots$.
The pushforward map $\tilde{B} \to \overline{X}$ factors through the inclusion of $\tilde{B}$ in $\C$. Using this fact and the push-pull formula on $\C$, the pushforward of $\psi_{k+1}^i$ along \eqref{btpush} is $\pi_*(\psi_{k+1}^i\cdot [\tilde{B}])$.
In particular, 
studying the diagram in Figure 1, we see that the kernel of the composition
\begin{equation} \label{comp} \qq[\eta, \kappa_1, \kappa_2, \ldots, \psi_1, \ldots, \psi_k] \to W_*H^*(\overline{X}) \to W_*H^*(X)
\end{equation}
is the ideal generated by $[A_1], \ldots, [A_k]$ and $\pi_*(\psi_{k+1}^i \cdot [\tilde{B}])$ for $i = 0, 1, 2, \ldots$. Equations \eqref{Aclass} and \eqref{notice} imply that
\begin{equation}
\label{eq:modAB}
\frac{\qq[\eta, \kappa_1, \kappa_2, \ldots, \psi_1, \ldots, \psi_k]}{\langle [A_1], \ldots, [A_k], \pi_*(\psi_{k+1}^i \cdot [\tilde{B}])_{i\geq 0} \rangle} \cong \begin{cases}
\qq[\eta] & \text{if $m_i \neq -\ell$ for all $i$} \\
\qq[\psi_{i_1}, \ldots, \psi_{i_s}] & \text{if $\{i: m_{i} = -\ell\} = \{i_1, \ldots, i_s\}$.}
\end{cases} 
\end{equation}
Our calculations establish that in degrees $* \leq \min\{\mathsf{i}^\ell_r(m), \mathsf{s}^\ell_r(m+2)+2\}$, the previously known relations \eqref{iota} in fact generate all relations in the tautological ring.

Finally if $* \leq \mathsf{s}^\ell_r(m)$, then the first map in \eqref{comp} is surjective by Lemma \ref{lem:stab}. The second is surjective since $X \subset \overline{X}$ is open, so the composition is also surjective.
\end{proof}

\section{The stable cohomology of $X^{\ell}(m_1, \ldots, m_k)$} \label{sec:stab}

In this section, we use the localization long exact sequence to study the full cohomology of $X^\ell(m_1, \ldots, m_k)$.
The closed complement of
$X^\ell(m_1,\ldots,m_k) \subset \overline{X}^\ell(m_1,\ldots,m_k)$ admits a stratification according to the zero type of the associated divisor. Given tuples of nonnegative integers $\vec{a} = (a_1, \ldots, a_k)$ and $\vec{b} = (b_1, \ldots, b_s )$ satisfying 
\[\sum_{j=1}^k (m_j + a_j) + \sum_{j= 1}^s b_j(j+1) \leq \ell(2g - 2),\]
we define $S(\vec{a},\vec{b}) \subset \overline{X}^{\ell}(m_1, \ldots, m_k)$ to be the locally closed boundary stratum 
of divisors with zero type
\[(m_1 + a_1, \ldots, m_k + a_k; 2^{b_1}, 3^{b_2}, \ldots, (s+1)^{b_s}).\]
This means that there is a zero of order $m_j + a_j$ at the $j$th marked point for $j = 1, \ldots, k$, there are $b_j$ zeros of order $(j+1)$ for $j=1,\ldots, s$, and the rest of the zeros are simple. Only the first $k$ zeros are labeled.
If all $m_i > -\ell$, the codimension of $S(\vec{a},\vec{b})$ in $\overline{X}^\ell(m_1, \ldots, m_k)$ is 
\[c = \sum_{j=1}^k a_j + \sum_{j=1}^s j b_j.\] 
Note that $S(\vec{0}, \varnothing) = X^\ell(m_1, \ldots, m_k)$.

Let $\overline{S}(\vec{a},\vec{b})$ be the closure of $S(\vec{a},\vec{b})$ in $\overline{X}^\ell(m_1, \ldots, m_k)$.
We call a class of the form 
\[[\overline{S}(\vec{a},\vec{b})] \cdot \eta^e \in H^{2c+2e}(\overline{X}^\ell(m_1, \ldots, m_k))\]
an \emph{$\eta$-decorated boundary stratum}.
Finally, let us define
\[d(i) \colonequals \dim \qq[\kappa_1,\kappa_2,\ldots, \psi_1,\ldots,\psi_k,\eta]_i\] 
to be the number of $\kappa,\psi, \eta$ monomials in degree $i/2$ if $i$ is even, and $0$ if $i$ is odd. 
\begin{lem}  \label{dcount}
The number of $\eta$-decorated boundary strata in degree $2i$ is equal to $d(2i)$.
\end{lem}
\begin{proof}
The $\eta$-decorated boundary strata of codimension $i$ are labeled by tuples of nonnegative integers $(a_1, \ldots, a_k, b_1, \ldots, b_s, e)$ such that
\[\sum_{j=1}^k a_j + \sum_{j = 1}^s jb_j + e = i.\]
Such data also labels the 
$\kappa, \psi, \eta$ monomials in degree $2i$ by assigning $(a_1, \ldots, a_k, b_1, \ldots, b_s, e)$ to the monomial
\[\prod_{j=1}^k \psi_j^{a_j} \cdot \prod_{\ell = 1}^s\kappa_\ell^{b_\ell} \cdot \eta^e. \qedhere\]
\end{proof}

Notice that each locally closed boundary stratum $S(\vec{a},\vec{b})$ of $\overline{X}^\ell(m_1, \ldots, m_k)$ can be realized as a finite group quotient of some $X^{\ell}(m_1', \ldots, m_{k'}')$. We take advantage of this inductive structure together with Lemma \ref{lem:stab} to prove results about the cohomology of the $X^\ell(m_1, \ldots, m_k)$.

\begin{lem} \label{12} 
Let $m_1, \ldots, m_k > -\ell$ and let $m= m_1 + \cdots + m_k$. Then
the map
\begin{equation} \label{qe} \qq[\eta] \to H^*(X^\ell(m_1,\ldots,m_k))
\end{equation}
is an isomorphism in degrees $i \leq \min\{\frac{2}{3}g - \frac{5}{3}, \frac{1}{2}(\ell(2g-2) - m - g+\delta_{1\ell})\}$. Moreover, if $i$ is odd, then the $\eta$-decorated boundary strata are independent in $H^{i+1}(\overline{X}^\ell(m_1, \ldots, m_k))$.
\end{lem}

\begin{proof}
We will prove the result by induction on $i$. The base case $i = 0$ is immediate as 
\eqref{qe} sends $1$ to the fundamental class which spans $H^0(X^{\ell}(m_1,\ldots,m_k))$.
Note that the hypotheses on $i$ imply that $i \leq \min\{\mathsf{s}^\ell_r(m)-1, \mathsf{i}^\ell_r(m), \mathsf{s}^\ell_r(m+2)+2\}$, so injectivity of \eqref{qe} follows from Theorem \ref{purewt}.
It remains to prove surjectivity of \eqref{qe} in the specified range of degrees.

\medskip
\textbf{Inductive step for $i$ odd.}
Assume $i \geq 1$ is odd and that we have proved the lemma for all smaller $i$. Since $i$ is odd, our goal is to show that $H^i(X^\ell(m_1, \ldots, m_k)) = 0$.
Since $i \leq \mathsf{s}^\ell_r(m)$, Theorem \ref{stab} implies that $H^i(\overline{X}^\ell(m_1,\ldots,m_k)) = 0$. Thus, it suffices to show that the restriction map
$H^i(\overline{X}^\ell(m_1,\ldots,m_k)) \to H^i(X^\ell(m_1,\ldots,m_k))$ is surjective.
Letting $\overline{X}^\ell_c(m_1, \ldots, m_k)$ denote the complement of all boundary strata of codimension $\geq c$, we can factor this restriction map
\begin{align*} H^{i}(\overline{X}^\ell(m_1,\ldots,m_k)) = H^{i}(\overline{X}^\ell_{(i+3)/2}(m_1,\ldots,m_k)) \to
H^{i}(\overline{X}^\ell_{(i+1)/2}(m_1,\ldots,m_k))
\to \cdots \\
\cdots \to H^{i}(\overline{X}^\ell_{2}(m_1,\ldots,m_k)) \rightarrow 
H^{i}(\overline{X}^\ell_{1}(m_1,\ldots,m_k)) =
H^{i}(X^\ell(m_1,\ldots,m_k)).
\end{align*}
The first equality above holds by localization and the last equality holds by definition. It therefore suffices to show that
\[H^i(\overline{X}^\ell_{c+1}(m_1,\ldots,m_k)) \to
H^i(\overline{X}^\ell_c(m_1,\ldots,m_k))\]
is surjective for each $1 \leq c \leq (i+1)/2$. 
Considering the localization long exact sequence for $\overline{X}_{c+1}^{\ell}(m_1,\ldots,m_k) \supset 
\overline{X}_c^{\ell}(m_1,\ldots,m_k)$, it suffices to show that
\begin{equation} \label{inj} \bigoplus_{\codim(S)= c} H^{i+1 - 2c}(S) \to H^{i+1}(\overline{X}_{c+1}^{\ell}(m_1,\ldots,m_k))
\end{equation}
is injective for each $1 \leq c \leq (i+1)/2$.
To establish injectivity of \eqref{inj}, we first prove the following claim.

\smallskip
\textit{Claim:} For each locally closed boundary stratum $S = S(\vec{a},\vec{b}) \subset \overline{X}^{\ell}(m_1,\ldots,m_k)$ of codimension $1 \leq c \leq (i+1)/2$, we have that $H^{i+1-2c}(S) = \qq \cdot \eta^{(i+1)/2 - c}$.

\smallskip
\textit{Proof:}
Each locally closed boundary stratum $S$ is a finite group quotient of $X^{\ell}(m_1',\ldots,m_{k'}')$ for some $m' = m_1' + \cdots + m_{k'}' \leq m + 2c$. The maximum $m'$ is obtained in the case when $(m_1', \ldots, m_k') = (m_1, \ldots, m_k,2,\ldots, 2)$, i.e.\ when $c$ pairs of simple zeros collide.
Moreover, 
 $i-2c+1 < i$ and we have
 \begin{align} \label{i2c} i-2c+1 &\leq \tfrac{1}{2}(\ell(2g - 2) - g - m + \delta_{1\ell}) - 2c + 1 \\
&\leq \tfrac{1}{2}(\ell(2g - 2) - g - m' + \delta_{1\ell}) - c + 1 \\
& \leq \tfrac{1}{2}(\ell(2g - 2) - g - m' + \delta_{1\ell}). \label{last}
\end{align}
Hence, the induction hypothesis
tells us that 
\[H^{i-2c+1}(X^\ell(m_1',\ldots,m_{k'}')) = \qq \cdot \eta^{(i+1)/2 - c}.\]
As the class $\eta$ is invariant under the symmetric group action permuting any subset of the markings, the claim follows.

\smallskip
We now prove injectivity of \eqref{inj}.
Since $H^{i - 2c + 1}(S) = \qq \cdot \eta^{(i+1)/2-c}$ for $1 \leq c \leq (i+1)/2$, the kernel of
\begin{equation} \label{resonemore} H^{i+1}(\overline{X}^\ell(m_1, \ldots, m_k)) \to H^{i+1}(X^\ell(m_1, \ldots, m_k))
\end{equation}
has dimension at most
the number of $\eta$ decorated boundary strata in $H^{i+1}(\overline{X}^\ell(m_1, \ldots, m_k))$ strictly supported on the boundary, and if equality holds then \eqref{inj} is injective for each $1 \leq c \leq (i+1)/2$. Counting the decorated boundary strata as in Lemma \ref{dcount}, 
this upper bound on the dimension of the kernel of \eqref{resonemore} is $d(i+1) - 1$. Moreover, since $i\leq \mathsf{i}^\ell_r(m)$,  Lemma \ref{lem:stab} implies that the $\eta, \kappa, \psi$ monomials span a subspace of dimension $d(i+1)$.
Meanwhile, we know from \eqref{iota} that the image of the monomials in $\eta, \kappa, \psi$ under \eqref{resonemore}
has dimension at most $1$.
It follows that the dimension of the kernel of \eqref{resonemore} is exactly $d(i+1) - 1$, and hence \eqref{inj} is injective. This implies also that the fundamental classes of all $\eta$-decorated boundary strata are independent in $H^{i+1}(\overline{X}^\ell(m_1, \ldots, m_k))$.   This completes the inductive step when $i$ is odd.

\medskip
\textbf{Inductive step for $i$ even.}
Now assume $i \geq 2$ is even and we have proved the lemma for all smaller $i$. 
Since $i \leq \mathsf{s}^\ell_r(m)$, Lemma \ref{lem:stab} implies that $H^i(\overline{X}^{\ell}(m_1,\ldots,m_k))$ is generated by monomials in $\eta, \kappa, \psi$. By \eqref{iota}, all monomials in $\eta, \kappa, \psi$ restrict to powers of $\eta$. Hence,
surjectivity will follow from surjectivity of \[H^i(\overline{X}^\ell(m_1,\ldots,m_k)) \to H^i(X^\ell(m_1,\ldots,m_k)).\]
Repeatedly applying the localization long exact sequence, excising first the codimension $i/2$ boundary strata, then the codimension $i/2-1$ boundary strata etc., it suffices to show that $H^{i-2c+1}(S) = 0$ for each codimension $c$ locally closed stratum $S \subset \overline{X}^\ell(m_1,\ldots,m_k)$ with $1 \leq c \leq i/2$. Each codimension $c$ stratum $S$ is a finite group quotient of some $X^\ell(m_1',\ldots,m_{k'}')$ with $m' = m_1' + \cdots + m_{k'}' = m + c + k' - k \leq m + 2c$. The maximum is obtained when $(m_1', \ldots, m'_{k'}) = (m_1, \ldots, m_k, 2, \ldots, 2)$, i.e. when $c$ pairs of simple zeros collide.
Now, $i-2c+1 < i$ and \eqref{last} holds,
so the induction hypothesis
tells us that $H^{i-2c+1}(X^\ell(m_1',\ldots,m_{k'}')) = 0$, and consequently all $H^{i-2c+1}(S) = 0$.
\end{proof}

\begin{rem}
\label{rem:gluing}
Up to labeling some of the simple zeros, the (tautological part of) cohomological stability can also be interpreted geometrically, by gluing a fixed  differential in low genus to a stratum of holomorphic differentials such that the resulting nodal differentials are smoothable and that $\eta$ pulls back to $\eta$ through the associated boundary inclusion map. 

For example, fix an elliptic curve $E$ with a meromorphic $\ell$-canonical divisor $\sum_{i=1}^{2\ell+1} z_i -(2\ell+1)p$. Given a differential in 
$\P^{\ell}(m_1, \ldots, m_k, 1^{\ell(2g-2)-m})$, identify one of the simple zeros with $p$ to form a node, i.e. $E$ is attached as an elliptic tail. Then, we obtain an inclusion from the stratum 
$\P^{\ell}(m_1, \ldots, m_k, 1^{\ell(2g-2)-m})$ in genus $g$ to the boundary of the stratum $\P^{\ell}(m_1, \ldots, m_k, 1^{\ell (2g)-m})$ in genus $g+1$. We can use the multi-scale compactification of strata of differentials (see \cite{BCGGM1, BCGGM2, CMZarea}) to check that the nodal differentials constructed this way can be smoothed into the interior of the stratum. Moreover, $\eta$ pulls back to $\eta$ through this inclusion map. 

Alternatively, fix a rational curve $R$ with a meromorphic $\ell$-canonical divisor
$\sum_{i=1}^{2\ell+2} z_i - (2\ell+1)(p_1+p_2)$. Given a differential in 
$\P^{\ell}(m_1, \ldots, m_k, 1^{\ell(2g-2)-m})$, identify two of its simple zeros with $p_1$ and $p_2$ respectively to form two nodes, i.e. $R$ is attached as a rational bridge. We thus obtain an inclusion from the stratum 
$\P^{\ell}(m_1, \ldots, m_k, 1^{\ell(2g-2)-m})$ in genus $g$ to the boundary of the stratum $\P^{\ell}(m_1, \ldots, m_k, 1^{\ell (2g)-m})$ in genus $g+1$, where $\eta$ pulls back to $\eta$ through the inclusion. 

If we do not want to distinguish any unordered simple zeros as attaching points, we can fix an elliptic curve with a meromorphic $\ell$-canonical divisor $\sum_{i=1}^{2\ell} z_i + m_1 p'_1 -(2\ell+m_1)p$ and glue $p$ to the original marked zero $p_1$ of order $m_1$ to form a smoothable nodal differential. Alternatively, we can add a plain marked point $p_0$, i.e. consider the stratum $\P^{\ell}(0, m_1, \ldots, m_k, 1^{\ell(2g-2)-m})/\mathbb{S}_{\ell(2g - 2) - m}$. Then we glue a fixed elliptic curve with a meromorphic $\ell$-canonical divisor $\sum_{i=1}^{2\ell} z_i -2\ell p$ by identifying $p$ with $p_0$ to form a smoothable nodal differential. In these constructions, $\eta$ also pulls back to $\eta$ through the corresponding boundary inclusion.  
\end{rem}

For meromorphic differentials, in general, we do not expect the stable cohomology of $X^\ell(m_1, \ldots, m_k)$ to be freely generated by tautological classes, as the following example shows.

\begin{example} \label{ex:mero}
We claim that
\[H^1(\P^1(-1, -1, -1, 1^{2g+1})/\mathbb{S}_{2g +1}) \neq 0\]
for all $g \geq 3$.
The compliment of $X^1(-1,-1,-1) \subset \overline{X}^1(-1,-1,-1)$ has four components:
$A_1,A_2,A_3$ and $\overline{B}$, as defined in Section \ref{sec:pure}. Now consider the long exact sequence
\begin{align*} H^1(\overline{X}^1(-1,-1,-1)) \to H^1(X^1(-1,-1,-1)) \\
\to H^0(A_1 \cup A_2 \cup A_3 \cup \overline{B}) \to H^2(\overline{X}^1(-1,-1,-1)) \to H^2(X^1(-1,-1,-1)).
\end{align*}
We have $H^1(\overline{X}^1(-1,-1,-1)) = 0$ and $\dim H^0(A_1 \cup A_2 \cup A_3 \cup \overline{B}) = 4$. Meanwhile, for $g \geq 3$, we have $H^2(\overline{X}^1(-1,-1,-1)) = \langle \eta, \kappa_1, \psi_1, \psi_2, \psi_3 \rangle$ and $H^2(X^1(-1,-1,-1)) = \langle \psi_1, \psi_2, \psi_3 \rangle$ by \eqref{eq:modAB}. It follows that $H^1(X^1(-1,-1,-1))$ is $2$-dimensional.
\end{example}

\section{Upper bounds on the vanishing degree of $\eta$} \label{sec:upper}

We now turn to the question of where relations first appear in the tautological ring.
Our strategy to show that \eqref{t1} has a non-trivial kernel is to take a known relation among the kappa classes on $\M_g$ in degree $a = \lfloor g/3 \rfloor + 1$ and pull it back to strata of differentials. 
This relation in degree $a$ is given by \cite[Proposition 1.7]{Ionel2} (see also \cite[Remark 5.11]{Mo03} for the version in cohomology). We determine a necessary and sufficient condition for the pullback of this relation to $\P(\mu)$ to be non-trivial, giving rise to a sufficient condition for Conjecture \ref{conj:upperbound} to hold.
To state the result, let
\[\mathsf{C}(t) \colonequals \sum_{k=0}^\infty \frac{(6k)!}{(3k)!(2k)!}\left(\frac{t}{72}\right)^k.\]
The shape of the relation depends on the residue of $g$ mod $3$, so we consider two cases.
\begin{prop} \label{g02}
Let $g = 0, 2 \pmod 3$, let $a = \lfloor g/3 \rfloor + 1$, and let $\mu = (m_1, \ldots, m_n)$ be a partition of $2g - 2$. If the coefficient of $t^a$ is non-vanishing in \begin{equation} \label{testmu02} \frac{\prod_{i=1}^n \mathsf{C}(\frac{t}{m_i + 1})}{\mathsf{C}(t)^{2g-2+n}},
\end{equation}
then $\eta^a = 0$, i.e.\ Conjecture \ref{conj:upperbound} holds.
\end{prop}
\begin{proof}
Following \cite{Ionel2}, we define integers $c_k$ by
\[\exp\left(\sum_{k=1}^\infty c_k t^k\right) = \mathsf{C}(t). \]
If $g = 0, 2 \pmod 3$, then by \cite[Proposition 1.7]{Ionel2}, the coefficient of $t^a$ in 
\[\exp\left(-\sum_{j=1}^\infty c_j \kappa_j t^j\right)\]
is a relation in $R^a(\M_g)$. We are interested in pulling this relation back to $\P(\mu)$ along the composition 
\[\P(\mu) \xrightarrow{\iota} \M_{g,n} \xrightarrow{f} \M_g.\] 
First recall that by \cite[Equation 1.5]{AC}, we have
\[f^*\kappa_j = \kappa_j - \sum_{i=1}^n \psi_i^j. \]
Applying $\iota^*$ and using \eqref{iota}, we find
\begin{align}
\label{eq:kappa-eta}
\iota^*f^*\kappa_j = \left(2g - 2 + n - \sum_{i=1}^n \frac{1}{(m_i + 1)^j}\right) \eta^j.
\end{align}
In particular,
\begin{align*} \iota^*f^* \exp\left(-\sum_{j=1}^\infty c_j \kappa_j t^j\right) &= \exp\left(-\sum_{j=1}^\infty c_j (2g - 2 + n) (\eta t)^j\right) \prod_{i=1}^n\exp\left(\sum_{j=1}^\infty c_j\left(\tfrac{\eta t}{m_i+1}\right)^j\right) \\
&= \frac{\prod_{i=1}^n \mathsf{C}(\frac{\eta t}{m_i + 1})}{\mathsf{C}(\eta t)^{2g-2+n}}.
\end{align*}
The coefficient of $t^a$ above is $\eta^a$ times the coefficient of $t^a$ in \eqref{testmu02}. If the coefficient is non-vanishing, then since pullback is a ring homomorphism, we conclude $\eta^a = 0$.
\end{proof}

\begin{prop} \label{g1}
Let $g = 1 \pmod 3$, let $a = \lfloor g/3 \rfloor + 1$ and let $\mu = (m_1, \ldots, m_n)$ be a partition of $2g - 2$. If the coefficient of $t^a$ is non-vanishing in 
\begin{align} \label{testmu} 
&\frac{\prod_{i=1}^n \mathsf{C}(\frac{t}{m_i + 1})}{\mathsf{C}(t)^{2g-2+n}}\left( 1 - 2t\left(2g - 2 + n - \sum_{i=1}^n \frac{1}{m_i+1}\right) 
\right. \\
& \qquad \qquad \qquad \qquad \left. - 12 t^2 
\left((2g - 2 + n)\frac{\mathsf{C}'(t)}{\mathsf{C}(t)} - \sum_{i=1}^n\frac{\mathsf{C}'(\frac{t}{m_i+1})}{(m_i+1)^2\mathsf{C}(\frac{t}{m_i+1})} \right)\right), \notag
\end{align}
then $\eta^a = 0$, i.e.\ Conjecture \ref{conj:upperbound} holds.
\end{prop}
\begin{proof}
If $g = 1 \pmod 3$, then \cite[Proposition 1.7]{Ionel2}  says that the coefficient of $t^a$ in 
\begin{equation} \label{newrel} \exp\left(-\sum_{j=1}^\infty c_j \kappa_j t^j\right)\left(1 - 2\kappa_1t - 12 \sum_{j=1}^\infty j c_j \kappa_{j+1}t^{j+1}\right)
\end{equation}
is a relation in $R^a(\M_g)$.
Note that
\[\sum_{j=1}^\infty jc_j t^{j+1} = t^2 \sum_{j=1}^\infty j c_j t^{j-1} = t^2 \frac{d}{dt} \sum_{j=1}^\infty c_j t^j = t^2 \frac{d}{dt} \log \mathsf{C}(t) = t^2\frac{\mathsf{C}'(t)}{\mathsf{C}(t)}.\]
Using this, when we apply $f^*\iota^*$ to the second term of \eqref{newrel}, we obtain
\begin{align*}
&1 - 2\left(2g - 2 + n - \sum_{i=1}^n \frac{1}{m_i+1}\right)(\eta t) - 12 \sum_{j=1}^\infty j c_j \left(2g - 2 + n - \sum_{i=1}^n \frac{1}{(m_i+1)^{j+1}}\right)(\eta t)^{j+1} \\
&
\qquad = 1 - 2\left(2g - 2 + n - \sum_{i=1}^n \frac{1}{m_i+1}\right)(\eta t) - 12 (\eta t)^2 
\left((2g - 2 + n)\frac{\mathsf{C}'(\eta t)}{\mathsf{C}(\eta t)} \right. \\
&\qquad \qquad \qquad \qquad \qquad \qquad \qquad \qquad \qquad \qquad \qquad \qquad \qquad  - \left. \sum_{i=1}^n\frac{\mathsf{C}'(\frac{\eta t}{m_i+1})}{(m_i+1)^2\mathsf{C}(\frac{\eta t}{m_i+1})} \right) \qedhere
\end{align*}
\end{proof}

\begin{proof}[Proof of Theorem \ref{thm:upperbound}] 
By Propositions \ref{g02} and \ref{g1}, to prove that Conjecture \ref{conj:upperbound} holds for a given $g$ and $\mu$, it suffices to show that the coefficient of $t^a$ in an appropriate power series is non-vanishing.
Using Sage, we have checked the coefficient of $t^a$ is non-vanishing in each relevant power series for any partition of $g$ for $g \leq 30$. To speed up the calculations, we worked over $\mathbb{F}_p$ for some prime $p$ that does not divide any denominators of coefficients of $\mathsf{C}(t)$ in degrees less than or equal to $a$. Indeed, if the coefficient is non-vanishing over $\mathbb{F}_p$, then the original coefficient over $\qq$ is also non-vanishing.
For some $\mu$, the coefficient of $t^a$ in the relevant power series vanishes over $\mathbb{F}_p$. For such $\mu$, we performed the calculation again over the finite field of the next prime order, and continued checking over larger finite fields until the coefficient was seen to be non-vanishing.
\end{proof}

\begin{rem}
\label{rem:conj-ell}
 Consider the stratum of $\ell$-differentials $\P^{\ell}(m'_1,\ldots, m'_n)$, where $m'_1 + \cdots + m'_n = \ell(2g-2)$ and $m'_i\neq -\ell$ for all $i$. Then by \cite[Proposition 2.1]{ChenTauto}, 
 \eqref{eq:kappa-eta} becomes 
 $$\iota^{*}f^{*}\kappa_j = \left( 2g-2+n - \sum_{i=1}^n \frac{1}{\left(\displaystyle\frac{m'_i}{\ell} + 1\right)^j} \right)\left(\frac{\eta'}{\ell}\right)^j,$$
 where $\eta'$ is the tautological line  bundle class on $\P^{\ell}(m'_1,\ldots, m'_n)$. 

Let $m_i = m'_i/\ell$ and $\eta = \eta' / \ell$. Then $(m_1, \ldots, m_n)$ becomes a ``partition'' of $2g-2$, where $m_i \in \mathbb Q$ and $m_i\neq -1$. Enlarging the scope of Propositions~\ref{g02} and~\ref{g1} to the case $m_i \in \mathbb Q\setminus \{-1\}$, Conjecture~\ref{conj:upperbound} can be reduced to checking the nonvanishing of the coefficient of the same power series, setting $m_i = m_i'/\ell$.

Finally, if some $m'_i = - \ell$, then $\eta' = 0$. Instead, in this case we should ask about the vanishing of the monomials of $\psi_i$ associated to the poles of order $\ell$. Notice that, in the case $\eta' = 0$, \eqref{iota} shows that the pullback of any relation among $\kappa$ classes is trivial. Thus, some other method will be needed to find relations among the $\psi_i$.
\end{rem}

\section{Hyperelliptic differentials and varying strata} \label{sec:varying}

In this section, we prove Theorem~\ref{thm:varying} which is reformulated as Proposition~\ref{prop:varying} below. We first set up some notation. 

Let $\nu = (2k_1, \ldots, 2k_m, 2\ell_1 - 1, \ldots, 2\ell_n-1)$ be a given zero and pole type of {\em quadratic differentials} (i.e. possibly meromorphic sections of $K^{\otimes 2}$) in genus $0$, where $k_i > 0$, $\ell_j \geq 0$, and 
$\sum_{i=1}^m 2k_i + \sum_{j=1}^{n} (2\ell_j-1) = -4$.  
Let $\P^2(\nu)$ be the stratum of quadratic differentials of type $\nu$, up to projectivization. For a quadratic differential $(\p^1, q) \in \P^2 (\nu)$, let $\pi\colon C\to \p^1$ be the canonical double cover branched at the $n$ 
odd-order zeros and poles of $q$, and $\pi^{*}q = \omega^2$, where $\omega$ is a holomorphic abelian differential on $C$ of zero type 
$$\mu = (k_1, k_1, \ldots, k_m, k_m,  2\ell_1, \ldots, 2\ell_n).$$ 
 This construction lifts $\P^2 (\nu)$ into the stratum $\P^1(\mu)$ as a subvariety, which we call the {\em locus of hyperelliptic differentials} in $\P^1(\mu)$. We refer to \cite[Section 2]{BCGGMk} for an introduction to the geometry of $\ell$-differentials and canonical $\ell$-covers for general $\ell$.  

\begin{prop}
\label{prop:varying}
Let $k_i > 0$ be odd for $1\leq i \leq m$, let $\ell_j > 0$ for $1\leq j\leq n_{+}$, and let $g = 1 + \sum_{i=1}^m k_i + 
\sum_{j=1}^{n_{+}} \ell_j$. Then for $m \geq 4$ and sufficiently large $g$, the stratum   
$$\P^1(k_1, k_1, \ldots, k_m, k_m, 2\ell_1, \ldots, 2\ell_{n_{+}})$$ is varying and $\eta \neq 0$ in this case.  \end{prop}

\begin{proof}
Let $n_0 = 2g+2 - n_{+}$. Consider 
$$\mu = (k_1, k_1, \ldots, k_m, k_m,  2\ell_1, \ldots, 2\ell_{n_{+}}, 0^{n_0}),$$
where $n = n_0 + n_{+}$ as in the preceding notation. 
Let $\T$ be a Teichm\"uller curve in the locus of hyperelliptic differentials in $\P^1(\mu)$ that lifts from the corresponding stratum of quadratic differentials 
$\P^2(\nu)$ in genus $0$.  
The area Siegel--Veech constant of $\T$ satisfies  
$$ \pi^2 \cdot c_{\area}(\T) =  1 - \frac{m + n}{2} + \sum_{i=1}^m \frac{1}{2k_i+2} + \sum_{j=1}^n \frac{1}{2\ell_j+1}, $$
see \cite[Corollary 1 and Theorem 1]{EKZ}. Additionally, the large genus asymptotic of $c_{\area}$ for generic holomorphic abelian differentials in non-hyperelliptic strata is 
$$ \lim_{g\to\infty} c_{\area}(\mu)^{\nonhyp} = \frac{1}{2}, $$
see \cite[Theorem 1.5]{CMSZ}, which can be approximated by using Teichm\"uller curves generated by large degree branched covers of elliptic curves, see \cite[Appendix A]{ChenRigid}. 

Since $\sum_{i=1}^m 2k_i + \sum_{j=1}^{n_{+}} (2\ell_j - 1) = n_0 - 4$, we have $ n_0 \geq 4 + 2m + n_{+}$, and hence  
\begin{eqnarray*}
\pi^2 \cdot c_{\area}(\T) \ge 1 - \frac{m+n_0+n_{+}}{2} + n_0   \geq  \frac{6+m}{2}.  
\end{eqnarray*}
For $m \geq 4$, the above implies 
$$\pi^2 \cdot c_{\area}(\T) \geq 5 > \frac{\pi^2}{2} = \pi^2 \lim_{g\to\infty} c_{\area}(\mu)^{\nonhyp}.$$
Therefore, $\P^1(\mu)$ is varying for sufficiently large $g$, and $\eta \neq 0$ in $A^1(\P^1(\mu))$ follows from \cite[Remark 3.1]{ChenNonvarying}. 
\end{proof}

\bibliographystyle{alpha}
\bibliography{biblio}
\end{document}